\documentclass[a4paper,11pt]{amsart}

\usepackage{amsmath,amssymb,amsfonts,latexsym}
\usepackage{resizegather}
\usepackage{graphicx}
\newtheorem{theorem}{Theorem}[section]

\newtheorem{corollary}[theorem]{Corollary}
\newtheorem{prop}[theorem]{Proposition}
\newtheorem{defn}{Definition}[section]

\begin{document}

\title[Special Intersection Graph in The Topological Graphs ]{Special Intersection Graph in The Topological Graphs}
\date{}
\maketitle
\begin{center}
\author{ Ahmed A. Omran $^{1}$, Veena Mathad$^{2}$, Ammar Alsinai$^{3}$ , Mohammed A. Abdlhusein $^{4}$.  \\
\small $^{2,3}$Department of Studies in Mathematics, University of Mysore, \\
\small Manasagangotri,  Mysuru - 570 006, India \\
E-mail:aliiammar1985@gmail.com $\&$ $veena_m$athad@rediffmail.com\\
\small $^{1}$ Department of  Mathematics, College of Education for Pure Science, University of Babylon, Babylon, Iraq\\
 E-mail: pure.ahmed.omran@uobabylon.edu.iq \\
 \small $^4$ Department  of  Mathematics, College of Education for Pure Science, University of Thi-Qar, Thi-Qar, Iraq.}\\
 E-mail:mmhd@utq.edu.iq

\end{center}
\begin{abstract}
 In this paper, new graphs $G_\tau=\left(V,E\right)$ are constructed from the discrete topological space $(X,\tau)\ $ . Several properties of this type of graphs are given such that: the clique number equals the number of elements in X also the number of pendants vertices, $G_\tau$ has no isolated vertices, the minimum degree in $G_\tau$ is one and maximum degree equal $n-1+\sum^{n-1}_{i=2}\binom{n-1}{i}$ , the minimum dominating set is determined and $\gamma(G_\tau)$ is evaluated for $G_\tau$ and for corona and join operations between to discrete topological graphs. At what matter $\beta\left(G_\tau\right)=\gamma(G_\tau)$ is discussed for $G_\tau$. Also that $G_\tau$ is proved a connected graph of order $2^n-2$ and it has no isolated vertex. Then,  rad $\ G_\tau$ and diam $\ (G_\tau)$ are evaluated.

\vspace{2mm}
\vspace{2mm}
\end{abstract}

\maketitle
\section{introduction}
 Topology is an important and basic branch of mathematics, which arose after the emergence of graphs which originated in 1736 when Leonhard Euler published his paper on the Seven Bridges of Königsberg. Topology has spread widely through theoretical and applied studies rather than graphs. In the recent period, a lot of studies have appeared linking the concepts of topology with graph theory. Some studies convert topology into a graph or vice versa, the vertex set of a graph has been used as an open set in topology \cite{A9}. In a different study, the minimal dominating set in a graph has been used as an open set in a topology\cite{A13,A14,A15}. In this paper, Let $X$ be a non-empty set, the discrete topological graph $G_\tau=\left(V,E\right)$ is a graph of the vertex set V=$\left\{A\ ;\ A\in\tau\ and\ A\neq\emptyset\ ,\ \ X\right\}$and the edge set $E=\left\{A\ B;\ A\cap B=\emptyset\right\}$. If the order of  $X$ is $n=2$ then $G_\tau$ will isomorphic to $K_2$ , if $n=3$ we prove $G_\tau\equiv{K_3\odot K}_1$. If $n\geq3$ then $G_\tau$ has induced complete graph of order $n$ and induced null graph of order n. Many bounds on the numbers mentioned above in addition to the maximum and minimum degrees are been determined. Also, some properties on the discrete topology are been proved. Moreover,  the domination number of some operation graphs is determined. Also, the radius and diameter of this graph are ben introduced. For more details, the reader can be seen in\cite{A10, A11, A12}.
 \section{Main results}
 \begin{defn}
   Let X be a non-empty set and $\tau$ be a discrete topology on $X$. The discrete topological graph denoted by $G_\tau=\left(V,E\right)$ is a graph of the vertex set $V=\left\{A;A\in\tau\ and\ A\neq\emptyset,\ X\right\}$ and the edge set $E=\left\{AB;A\cap B=\emptyset\right\}$.
 \end{defn}
 \begin{prop}\label{DD8}
 If $n=2$, then $G_\tau\equiv K_2$.
 \end{prop}
 \begin{proof}
 . Let $X=\left\{1,2\right\}$, then $\tau=\left\{\emptyset,\ X,\ \left\{1\right\},\ \left\{2\right\}\right\}$, so $V=\left\{\left\{1\right\},\ \left\{2\right\}\right\}$. It is obvious that this graph is isomorphic to the graph $K_2$, since the vertex $\left\{1\right\}$ is adjacent to the vertex $\left\{2\right\}$.
 \end{proof}
 \begin{prop}\label{EE8}
   . If $\left|X\right|=3$, then $G_\tau\equiv{K_3\odot K}_1$.
 \end{prop}
 \begin{proof}
 Let $X=\left\{1,2,3\right\}$, then $\tau=\left\{\emptyset,\ X,\ \left\{1\right\},\ \left\{2\right\},\left\{3\right\},\ \left\{1,2\right\},\left\{1,3\right\},\left\{2,\ 3\right\}\right\}$, so $V=\left\{\left\{1\right\},\ \left\{2\right\},\ \left\{3\right\},\ \left\{1,2\right\},\left\{1,3\right\},\left\{2,\ 3\right\}\right\}$. Let $u$ and $v$ be two vertices of singleton element, thus $u$ is adjacent to $v$. Therefore, the vertices that include a singleton element $\left\{\left\{1\right\},\ \left\{2\right\},\ \left\{3\right\}\right\}$ form a complete induced subgraph as shown in Figure\ref{AA8}. The remains vertices are $\left\{\left\{1,2\right\},\left\{1,3\right\},\left\{2,\ 3\right\}\right\}$, these vertices are pairwise joined and each one of them is adjacent to only one vertex from a singleton vertex as shown in Figure $2.1$. Thus, $G_\tau\equiv{K_3\odot K}_1$.
 \end{proof}
\begin{figure}[h]
	\centering
	\includegraphics[width=0.3\linewidth]{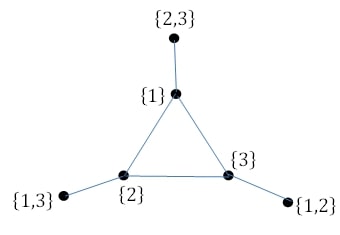}
	\caption{The graph $G_\tau\equiv{K_3\odot K}_1$.}
	\label{AA8}
\end{figure}
\begin{prop}\label{CC8}
Let $\left|X\right|=n$, then\\
\begin{enumerate}
  \item The clique number is $n$.
  \item The number of pendants vertices is $n$.
\end{enumerate}
\end{prop}
\begin{proof}
 {\bf({1})}:In the same manner in the previous proposition, the subset of the vertex set which contains only a singleton element makes an induced subgraph S isomorphic to a complete graph of order $n$. The remain vertices contain at least two elements say $\left\{a,b\right\}$,$\ a,b=1,2,\ldots,\ n$, then $S\cup\left\{a,b\right\}$ is not complete, since $\left\{a\right\},\ \left\{b\right\}\in S$ (as an example, see Figure 2.2).\\
 {\bf({2})}: Let $S$ be a set that contains all vertices which have $n-1$ elements, then each vertex from the set $S$ say $u$ is adjacent to one vertex say $v\ $ that has singleton element such that $\left\{u\right\}\cup\left\{v\right\}=V$, and it is obvious that $\left|S\right|=n$. Now, if there exists another pendant vertex say w, so the most number of elements of this vertex is $n-2$. Thus, there are two vertices of singleton element which do not belong to elements of the vertex $w$, these singleton vertices are adjacent to the vertex w and this is a contradiction with a pendant property (as an example, see Figure 2.2). Thus, the pendants vertices have only in the set $S$ and the required is done.
\end{proof}
\begin{figure}[h]
	\centering
	\includegraphics[width=0.7\linewidth]{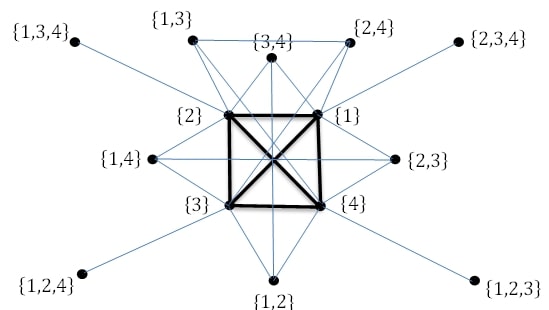}
	\caption{The graph $G_\tau~ graph whenever ~ \left|X\right|=4$.}
	\label{BB8}
\end{figure}
\begin{theorem}
 Let $\left|X\right|=n$, and $G_\tau $ be a discrete topological graph, then
 \begin{enumerate}
   \item $\delta\left(G_\tau\right)=1.$
   \item $\Delta(G_\tau)=n-1+\sum^{n-1}_{i=2}\binom{n-1}{i}$.
 \end{enumerate}

\end{theorem}
\begin{proof}
{\bf({Case 1})}:Since $G_\tau $ has no isolated vertex, then the required is obtained straightforwardly from proposition \ref{CC8}.\\
{\bf({Case 2})}:The maximum degree of the graph $G_\tau$ is founded in each vertex of the singleton element and these vertices make as an induced subgraph of order $n $ isomorphic to the complete graph. Thus, each vertex of them say $v $ is adjacent to all other vertices in that complete induced subgraph, so deg$\left(v\right)\geq n-1$ (as an example, see Figure \ref{BB8}). Moreover, the vertex $v$ is adjacent to all vertices that have two elements different from $v$ and the number of these vertices is $\binom{n-1}{2}$,\ \ if $n\geq3$ (as an example, see Figure \ref{BB8}). Again, the vertex $v$ is adjacent to all vertices that have three elements different from $v$ and the number of these vertices is $\binom{n-1}{3}$, if $n\geq3$ (as an example, see Figure \ref{BB8}) and so on…. The last step when the vertex $v$ is adjacent to the vertices have $n-1$ element and the number of these vertices is $\binom{n-1}{n-1}=1$ (as an example, see Figure \ref{BB8})Therefore,$\Delta(G_\tau)=n-1+\sum^{n-1}_{i=2}\binom{n-1}{i}$.
\end{proof}
\begin{theorem}\label{FF8}

\item Let $G_\tau$  be a discrete topological graph, where $\left|X\right|=n$, then, \\
  $$\gamma\left(G_\tau\right)=\left\{%
\begin{array}{ll}
    1,  \hbox{  if $n =2$ ;} \\
   n,  \hbox{  if $n > 2$ .} \\
\end{array}%
 \right.    $$
\end{theorem}
\begin{proof}
{{\bf Case 1}}: If n=2, then according to Proposition \ref{DD8}, $G_\tau\equiv K_2$, so the required is obtained.\\
{{\bf Case 2}}: i f$ n>2$, then according to Proposition \ref{CC8}, the number of pendants vertices is n and there is a support vertex adjacent to more than one pendant. Thus, $\gamma\left(G_\tau\right)\geq n$. Each support vertex has a singleton element, and dominates only one of the pendants vertices, so, the set of the support vertices say S dominates to all pendants vertices. Moreover, the $r$ remaining  vertices are not pendants or supports have the cardinal $\mu$ such that $2\le\mu\le n-2$, so each vertex of them is adjacent to at least one vertex of $S$, if not that means there is a vertex adjacent to all vertices that have singleton vertex, and this means again the cardinal of this vertex is equal to $n$ and this is a contract. Therefore, set $S$ is dominating set with minimum cardinality, and the result is obtained.
\end{proof}
\begin{theorem}\label{GG8}
Let $G_\tau $  be a discrete topological graph, where $\left|X\right|=n$ , then
$$\beta\left(G_\tau\right)= \sum^{{n-1}}_{i=\lfloor \frac{n}{2}\rfloor}\binom{n}{i}$$
\end{theorem}
\begin{proof}
  Let $D=\left\{S_i;i=n,\ n+1,\ldots,\ 2n-3\right\}$, where $S_i$ is the set of all vertices have $\lfloor \frac{i+1}{2}\rfloor$ elements. It is clear that $i>\frac{n}{2}\ \ \forall i$, thus $ S_i\cap S_j\neq\emptyset;\ \forall i,j=n,\ n+1,\ldots,\ 2n-3$. Therefore, set $D$ is independent. Now, to prove that the set $D$ is maximal independent, since the independent property is hereditary, thus, it is enough to prove that the set $D$ is $1-maximal$. Suppose that the set $d$ is not maximal, so there is a vertex $v;v\in V-D$ such that $D\cup\left\{v\right\}$  is an independent set. since $v\in V-D$, the vertex $v\ $ has $t$ element such that $t\le\frac{n}{2}$ that means there is at least one vertex in the set $D$ say $u$ such that $\left\{u\right\}\cap\left\{v\right\}=\emptyset$. This leads to the two vertices being adjacent to each other and this is a contract with the independent property. Therefore, the set $D$ is $1-maximal$ and this leads to being the set $D$ is maximal. Moreover, if there is another maximal independent set say $D_1$, then all vertices belonging in this set is belonging to $V-D$ and it is clear that there are many of vertices in the set $V-D$ are adjacent, so $\left|D_1\right|<\left|D\right|$. Thus, the set $D$ is the maximum independent set and $\beta\left(G_\tau\right)=\left|D\right|=\sum_{i=n}^{2n-3}\binom{n}{\lceil \frac{ i+1}{2}\rceil}.$
\end{proof}
\begin{corollary}
  Let $G_\tau$  be a discrete topological graph, where $\left|X\right|=n$, then
$\beta\left(G_\tau\right)=\gamma\left(G_\tau\right)$  $if and only if n=3$.
\end{corollary}
\begin{proof}
If $n=3$, then according to Proposition \ref{EE8}, $\ \tau=\left\{\emptyset,\ X,\ \left\{1\right\},\ \left\{2\right\},\left\{3\right\},\ \left\{1,2\right\},\left\{1,3\right\},\left\{2,\ 3\right\}\right\}$, so $V=\left\{\left\{1\right\},\ \left\{2\right\},\ \left\{3\right\},\ \left\{1,2\right\},\left\{1,3\right\},\left\{2,\ 3\right\}\right\}$.Moreover, $D=\left\{\left\{1\right\},\ \left\{2\right\},\ \left\{3\right\}\right\}$ is  a minimum dominating set by proof of Theorem \ref{FF8}, and the set $I=\left\{\ \left\{1,2\right\},\left\{1,3\right\},\left\{2,\ 3\right\}\right\}$ is the maximum independent set according to Theorem \ref{GG8}, as shown in Figure \ref{CC8}. Thus, $\beta\left(G_\tau\right)=\gamma\left(G_\tau\right)$.\\
Conversely, If $\beta\left(G_\tau\right)=\gamma\left(G_\tau\right)$, according to Theorem\ref{FF8}, $\gamma\left(G_\tau\right)=n$ and so by assumption $\beta\left(G_\tau\right)=n$. Moreover, the number of pendant vertices is $n$, so the maximum independent set contains only pendants vertices and this situation is only achieved when $n=3$ as shown in Figure \ref{CC8}. Therefore, the required is obtained.
\end{proof}
Let $\left|X\right|=5$, then \\
$\tau= \{  \emptyset,\ X,\ \left\{1\right\},\ \left\{2\right\},\left\{3\right\},\ \left\{4\right\},\ \left\{5\right\},  \ \left\{1,2\right\},$\\ $\left\{1,3\right\},\ \left\{1,\ 4\right\},\ \left\{1,\ 5\right\},\ \left\{2,\ 3\right\},\ \left\{2,\ 4\right\},\ \left\{2,\ 5\right\} $, \\ $\ \left\{3,\ 4\right\},\ \left\{3,\ 5\right\},\left\{4,\ 5\right\},\ \left\{1,2,\ 3\right\},\ \left\{1,2,\ 4\right\},\ \left\{1,2,\ 5\right\},$ \\ $\ \left\{1,3,\ 4\right\},\ \left\{1,3,\ 5\right\},\ \left\{1,4,\ 5\right\},\ \left\{2,3,\ 4\right\},\ \left\{2,3,\ 5\right\},\ \left\{2,4,\ 5\right\},$ \\ $ \ \left\{3,4,\ 5\right\},\left\{1,2,\ 3,4\right\},\ \left\{1,2,\ 3,5\right\},\ \left\{1,2,\ 4,5\right\},\ \left\{1,3,\ 4,5\right\},\ \left\{2,3,\ 4,5\right\} \}$\\
so $V= \{ \left\{1\right\},\ \left\{2\right\},\left\{3\right\},\ \left\{4\right\},\ \left\{5\right\},\ \left\{1,2\right\},\left\{1,3\right\},\ \left\{1,\ 4\right\},\ \left\{1,\ 5\right\},$ \\ $\ \left\{2,\ 3\right\},\ \left\{2,\ 5\right\},\ \left\{3,\ 4\right\},\ \left\{3,\ 5\right\},\left\{4,\ 5\right\},\ \left\{1,2,\ 3\right\},\ \left\{1,2,\ 4\right\},\ \left\{1,2,\ 5\right\},$ \\$ ,\ \left\{1,3,\ 4\right\},\ \left\{1,3,\ 5\right\},\ \left\{1,4,\ 5\right\},\ \left\{2,3,\ 4\right\},\ \left\{2,3,\ 5\right\},\ \left\{2,4,\ 5\right\},\ \left\{3,4,\ 5\right\},$ \\ $ \left\{1,2,\ 3,4\right\},\ \left\{1,2,\ 3,5\right\},\ \left\{1,2,\ 4,5\right\},\ \left\{1,3,\ 4,5\right\},\ \left\{2,3,\ 4,5\right\}$          \},and the graph as the following figure \\
$D=\left\{\left\{1\right\},\ \left\{2\right\},\left\{3\right\},\ \left\{4\right\},\ \left\{5\right\}\right\}, is the minimum dominating set and \gamma\left(G_\tau\right)=5$.\\
$I=\{ \left\{1,2,\ 3\right\},\ \left\{1,2,\ 4\right\},\ \left\{1,2,\ 5\right\},\ \left\{1,3,\ 4\right\},\ \left\{1,3,\ 5\right\},\ \left\{1,4,\ 5\right\},\ \left\{2,3,\ 4\right\},\ \left\{2,3,\ 5\right\},$\\ $\left\{2,4,\ 5\right\},\ \left\{3,4,\ 5\right\},\left\{1,2,\ 3,4\right\},\ \left\{1,2,\ 3,5\right\},\ \left\{1,2,\ 4,5\right\},\ \left\{1,3,\ 4,5\right\},\ \left\{2,3,\ 4,5\right\}  \}$\\
the maximum independent set and $\beta\left(G_\tau\right)=15$. $\delta\left(G_\tau\right)=1$.  $\Delta(G_{\tau})=15$. The clique is $K_5$ and the set of pendant vertices is\\
$ S=\{ \left\{1,2,\ 3,4\right\},\ \left\{1,2,\ 3,5\right\},\ \left\{1,2,\ 4,5\right\},\ \left\{1,3,\ 4,5\right\},\ \left\{2,3,\ 4,5\right\} \}$,\\
so $|S|=5$
\begin{figure}[h]
	\centering
	\includegraphics[width=0.9\linewidth]{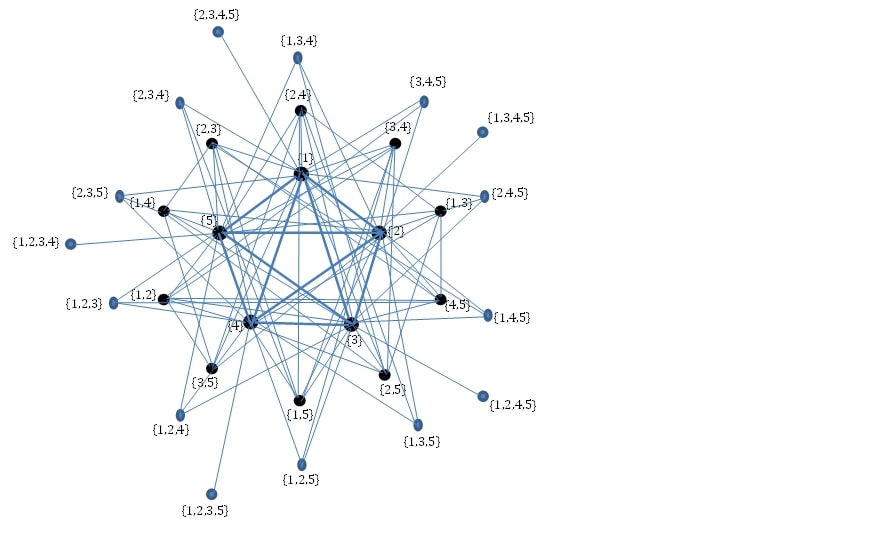}
	\caption{The discrete topological graph, where $\left|X\right|=5$.}
	\label{HH8}
\end{figure}
\begin{theorem}
Let $G_\tau$  be a discrete topological graph of a nonempty set $\ X$, then $G_\tau$ is a connected graph.
\end{theorem}
\begin{proof}
Let $v_1$ and $v_2$ are any two vertices in $G_\tau $, let $S$ be a set of all vertices that have singleton elements in $G_\tau$. There are three cases as follows:\\
{{\bf Case 1}}: If  $v_1\ ,\ \ v_2\in S$ then there is an edge $v_1\ v_2 in G_\tau$ since $G\left[S\right]=K_n$.\\
{{\bf Case 2}}: If $v_1\in S$  and $v_2\notin S$ then there is an edge between $v_2$ and at least one vertex from $S$ say $u$ (since the element in $u$ don't belong to $v_2 $ , so ${u\cap v}_2=\emptyset)$. Thus, $v_1-\ u-v_2$ is a path in $G_\tau$.\\
{{\bf Case 3}}: If  $v_{1\ },\ \ v_2\notin S $ then in a similar technique of case $2$, either there is a vertex $t\ \in S$ such that $v_1\ t and v_2\ t$ are two edges in $G_\tau$ , then  $v_1-\ t-v_2$ is a path in $G_\tau$ . Or there are two edges $v_1\ u$ and $v_2\ w$ in $G_\tau$ such that $u\ ,\ \ w\ \in S$. Thus,$ v_1-\ u-w-v_2$ is a path in $G_\tau$.  Hence, in all the above cases, there is a path between any two vertices in $G_\tau$ which proves it a connected graph.
\end{proof}
\begin{prop}
The order of discrete topological graph $G_\tau$ is $2^n-2$ if  $\left|X\right|=n$.
\end{prop}
\begin{proof}
Since $\tau$ is a power set containing all subsets of  $X$ which are $2^n$ sets, then $G_\tau$ has all elements of  $\tau$ unless $\emptyset$ and $X$ according to its definition.
\end{proof}
\begin{prop}
There is no isolated vertex in discrete topological graph $G_\tau$.
\end{prop}
\begin{proof}
Let  $u\in\tau$ then  $u\subseteq X $ so that $u^c\subseteq X$ and since $u\cap u^c=\emptyset$ , thus there is an edge between $u$ and $u^c$ which gives the result.
\end{proof}
\begin{prop}
Let $\left|X\right|=n$ and $G_\tau$ be a discrete topological graph, the $G_\tau$ has $N_n$ induced subgraph.
\end{prop}
\begin{proof}
Let $S$ be a set of all subsets of  $X$ that have $n-1$ elements from $X$. It is obvious $\left|S\right|=n$ and $S\subseteq\tau$. Let $u\ ,\ v\in S$ then there is at least one element belonging to $u\cap v$ . Hence, $u\cap v\neq\emptyset$ and there is no edge between any two vertices of  $S$ .  Thus, $G[S]$ is a null graph of order $n$ .
\end{proof}
\begin{theorem}
Let $G_\tau$ be a discrete topological graph of order $n$ defined on a set $X$. And $H_\tau$ be a discrete topological graph of order $m$ defined on a set $Y$, then
\begin{enumerate}
  \item $\gamma\left(G_\tau\bigodot H_\tau\right)=n $.
  \item $ \gamma\left(G_\tau+H_\tau\right)=\left\{%
\begin{array}{ll}
    \gamma\left(G_\tau\right),  \hbox{  if $n\le m$ ;} \\
   \gamma\left(H_\tau\right),  \hbox{  if $n>m$ .} \\
\end{array}%
 \right.    $
\end{enumerate}

\end{theorem}
\begin{proof}
{{\bf(1)}}: Let $v_i\in V(G_\tau)$, then $v_i$ is adjacent with all vertices of the $i^{th}$ copy of $H_\tau$ , thus  $v_i$ dominates all vertices of one copy of $H_\tau$, so  $v_i\in D$. Therefore, $D=V(G_\tau)$ is the minimum dominating set of  $G_\tau\bigodot H_\tau$ and  $\gamma\left(G_\tau\bigodot H_\tau\right)=n$.\\
{{\bf(2)}}: Let $n\le m$, since  $\gamma\left(G_\tau\right)=n$  and its dominating set contains only the sets of singleton elements in $G_\tau$ according to Theorem \ref{FF8} and since these singleton elements will adjacent with all vertices of $H_\tau$ according to corona operation definition. Then, the dominating set of $G_\tau+H_\tau$ is the same dominating set of $G_\tau$. Hence, $D\left(G_\tau+H_\tau\right)=D(G_\tau)$ and $\gamma\left(G_\tau+H_\tau\right)=\gamma\left(G_\tau\right)=n$. In a similar technique if $n>m$ we can prove $D\left(G_\tau+H_\tau\right)=D(H_\tau)$ and $\gamma\left(G_\tau+H_\tau\right)=\gamma\left(H_\tau\right)=m$ .
\end{proof}
\begin{theorem}
 Let $\left|X\right|=n (n\geq3)$ and $G_\tau$ be a discrete topological graph on $X$, then  ${rad\ (G}_\tau)=2$ and $diam{\ (G}_\tau)=3$.
\end{theorem}
\begin{proof}
Let $v\ \in V(G_\tau)$ such that it has a singleton element, then either$ d\left(v,\ w\right)=1$ if the vertex w of the singleton element or $\ v\ w\in E(G_\tau)$ according to Proposition\ref{CC8}. Or  $d\left(v,\ w\right)=2$  if  $w$ has more than one element and $v\ w\notin E(G_\tau)$ since $v-\acute{v}-w$ is a shortest $\left(u,v\right)-path$ in $G_\tau$ where $\acute{v}$ has singleton element. Thus, the eccentricity of every vertex that has a singleton element is $e\left(v\right)=2$ which is the minimum eccentricity among all vertices of $G_\tau$. Hence, ${rad\ (G}_\tau)=2$ .
Now, to evaluate diameter $G_\tau$ , let $w_1\ ,\ \ w_2\in V(G_\tau)$ such that these vertices has non-singleton elements and $w_1\ w_2\notin E(G_\tau)$ . Either $w_1$ and $w_2$ adjacent with the same vertex $v$ , so $d\left(v,\ w\right)=2$ . Or $w_1-u-v-w_2$ is the shortest  $\left(w_1,w_2\right)-$path of length $3$ in $G_\tau$ where $v$ and $u$ have singleton elements. Thus,$e\left(w_1\right)=e\left(w_2\right)=3$ which is the maximum eccentricity among all vertices of $G_\tau$. Hence, $diam{\ (G}_\tau)=3$.
\end{proof}
\begin{prop}
Let $\left|X\right|=n (n\geq3)$ and $G_\tau $ be a discrete topological graph on $X$, then every vertex has a singleton element is a cut vertex and every vertex has $n-1$ element is not cut vertex in $G_\tau$.
\end{prop}
\begin{proof}
Since every vertex of singleton element in $G_\tau$ is a subset from all vertices of $n-1$ elements unless one vertex. Then, there is only one edge between a vertex $u$ of singleton element and vertex $v$ of  $n-1$ elements where $u\nsubseteq v$. Hence, $G_\tau-u$ contains at least two components one of them is the isolated vertex $v$ . Therefore, $u$ is a cut vertex and $v$ is not cut vertex since it was pendant vertex in $G_\tau$
\end{proof}
\section*{Conclusions}
From, the results are mentioned above, the new graphs depending on the discrete topology are been introduced. Many properties of this graph are been proved. Moreover, the domination, independence, and clique numbers are been determined. Also,  the domination number for join and corona operation of two graphs are been discussed.


\end{document}